\documentclass[a4paper,10pt,oneside]{amsart}

\usepackage{amsmath,amssymb}
\usepackage[english]{babel}
\usepackage{url}
\usepackage{hyperref}
\usepackage{tikz}
\usetikzlibrary{arrows,positioning,calc,patterns}
\usepackage{tikz-cd}
\usepackage{comment}
\hypersetup{
	colorlinks,
	linkcolor={red!50!black},
	citecolor={blue!50!black},
	urlcolor={blue!80!black},
	breaklinks=true
}

\theoremstyle{plain}
\newtheorem{lemma}{Lemma}[section]

\newtheorem{theorem}[lemma]{Theorem}

\newtheorem*{theorem-nonum}{Theorem}

\theoremstyle{remark}

\theoremstyle{definition}

\title{No perfect triangle is isosceles}

\author[M. Makhul]{Mehdi Makhul}
\address[Mehdi Makhul]{Fakult\"{a}t f\"{u}r Mathematik Universit\"{a}t Wien, Austria}
\email{mmakhul@risc.jku.at}
\subjclass[2010]{52C10, 14G05} 
\keywords{Perfect triangle, Quartic Diophantine equation}

\begin{document}
	
\begin{abstract}
A \emph{perfect triangle} is a triangle with rational sides, medians, and
area. In this article,  we use a similar strategy due to Pocklington to
show that if $\Delta$ is a perfect triangle, then it cannot be an
isosceles triangle. It gives a partial answer to a question of Richard Guy, who asked whether any perfect triangles exist. No example has been found to date. It is widely believed that such a triangle does not exist.
\end{abstract}

\date{}
\maketitle
\section{Introduction}
\label{intr}
A \emph{median} of a triangle is the line segment joining a vertex to the midpoint of the opposite side. We say that a side, or a median, of a triangle is \emph{rational} if its length is a rational number. Finding a triangle with rational sides, medians and area was asked as an open problem by Richard Guy in \cite[D$21$]{Guy2004}. Such a triangle is called a  \emph{perfect triangle}. Various research has been done towards this question, but to date the problem remains unsolved (see, e.g., \cite{Shahrina2017} for a survey). If we do not require the area to be rational, there are infinitely many solutions. Euler gave a parametrization of such "rational triangles", in which "all three medians are rational", see \cite{Buchholz2002}, however there are examples of triangles with three integer sides and three integer medians that are not given by the Euler parametrization. Buchholz in \cite{Buchholz2002} showed that every rational triangle with rational medians corresponds to a point on a one parameter elliptic curve. Buchholz and Rathbun~\cite{Buchholz1997} also showed the existence of infinitely many \emph{Heron triangles} with two rational medians, where a Heron triangle is a triangle that has side lengths and area that are all rationals. 

Similar problems have been studied in other settings, and of particular relevance to this work is the existence of a \emph{perfect square triangle}, namely a triangle whose sides are perfect squares and whose angle bisectors are integers. Lucas~\cite{Luca2018}  has shown that the existence of a perfect square triangle is equivalent the existence of a \emph{perfect cuboid}, i.e. a rectangular box with all sides, face diagonals, and main diagonals are integers. Up to now, it is known \cite{Makhul2020} \cite{Zelator2008} that there is no perfect right triangle i.e a perfect triangle with a right angle. A current result in this direction is the following result~\cite{Makhul2020}. 
\begin{theorem}\label{thm:perfect-triangle}
Given $0<\theta < \pi$ where $\theta \not= \frac{\pi}{2}$, there are, up to similarity, finitely many Heron triangles with an angle $\theta$ and with two rational medians.
\end{theorem}
Recall that a triangle is called an \emph{isosceles} triangle if two of its sides have equal length. In this paper, we will explore the following result.
\begin{theorem}\label{thm:main}
No perfect triangle is isosceles.
\end{theorem}
In the next section, we will see the relation between the isosceles perfect triangle problem and a certain quartic Diophantine equation of the form $ax^4+bx^2y^2+cy^4$ and we use a similar strategy due to Pocklington \cite{Polcklington1912} \cite{Sierpinski1988}, who used to show that there are no four different squares which form an arithmetical progression. The solvability of the Diophantine equation $x^4+bx^2y^2+y^4$ has been considered by mathematicians from Fermat and Euler. There are nice results in this direction see, e.g., \cite{Brown1989} \cite{Cohn1994} \cite{Cohn1983}.

\subsection{From Perfect triangles to quartic Diophantine equation}
\label{sec:FPTQDE}
We now briefly describe our strategy for proving Theorem \ref{thm:main}. Let $\Delta$ be a rational isosceles triangle, with legs $A$ and base~$B$. Without loss of generality, we may assume~$A$ and $B$ are coprime positive integers. Let~$h$ and~$l$ denote the median lengths of $\Delta$; see Figure~\ref{fig:isose-angle}. By the formulas expressing medians in terms of edges,
we have
\begin{equation*}
\begin{aligned}
h^2&=A^2-(B^2/4), \\ 
4l^2&=2B^2+A^2.
\end{aligned}
\end{equation*}
Consider the change of variable $2l=w$, then we get $w^2=2B^2+A^2$. By Lemma~\ref{lm:1} we have $A=m^2-2n^2$ and $B=2mn$. In particular, since $(A,B)=1$ we have~$(m, n)=1$. For if $(m,n)=d>1$, then $m=dm_1$ and $n=dn_1$, which implies $A=d^2(m_1^2-2n_1^2)$ and $B=2d^2m_1n_1$, therefore $(A,B)\ge d$ which is a contradiction.\footnote{Since $(A,B)=1$ either both $m$ and $n$ are odd or $m$ odd $n$ even.} By substituting the expressions for $A$ and $B$ in the equation $h^2=A^2-\frac{B^2}{4}$ we obtain 
\[
h^2=(m^2-2n^2)^2-(mn)^2=m^4+4n^4-5m^2n^2.
\]
\begin{figure}
	\begin{tikzpicture}
	\filldraw [black] (0,0) node[anchor=west]{} circle (2pt);
	\filldraw [black] (4,0) node[anchor=west]{} circle (2pt);
	\filldraw [black] (2,3) node[anchor=west]{} circle (2pt);
	\filldraw [black] (2,0) node[anchor=west]{} circle (2pt);
	\filldraw [black] (1,1.5) node[anchor=west]{} circle (2pt);
	\filldraw [black] (3,1.5) node[anchor=west]{} circle (2pt);
	\draw[black, thick] (2,3) -- (4,0);
	\draw[black, thick] (0,0) -- (4,0);
	\draw[black, thick] (0,0) -- (2,3);
	\node[circle, anchor=north] (n1) at (2,0) {$B$};
	\node[circle, anchor=east] (n1) at (1,1.5) {$A$};
	\node[circle, anchor=west] (n1) at (3,1.75) {$A$};
	\node[circle, anchor=west] (n1) at (1,1.5) {$l$};
	\node[circle, anchor=west] (n1) at (2,0.5) {$h$};
	\node[circle, anchor=west] (n1) at (2.6,1.2) {$l$};
	\draw[gray, thick] (0,0) -- (3,1.5);
	\draw[gray, thick] (4,0) -- (1,1.5);
	\draw[gray, thick] (2,3) -- (2,0);
	\end{tikzpicture}
	\caption{An isosceles triangle with side lengths $A$, and $B$, and median lengths $l,h$.}\label{fig:isose-angle}
\end{figure}
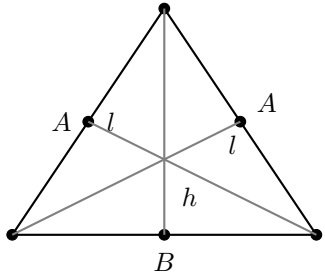
Therefore, the existence of an isosceles triangle implies the existence of a non-trivial solution of the equation $k^2=4t^4-5t^2s^2+s^4$ in positive integers. A solution $(t_0, s_0, k_0)$ of the latter equation is said to be non-trivial if~$t_0s_0k_0\not=0$.
\begin{theorem}\label{thm:main-equation}
The equation $k^2=4t^4-5t^2s^2+s^4$ has only trivial solutions.
\end{theorem}
We have to point out that the author has been informed that by substitution $T=t/s$, we obtain $4T^4 - 5T^2 + 1 = Z^2$, which is an elliptic curve of rank zero, and by using the theory of rational points of the elliptic curves we can show that this equation has only trivial solutions. However, we still
think that our elementary approach and the result could be of interest to researchers. We defer the proof of Theorem~\ref{thm:main-equation} to Section \ref{sec:main}. 
\section{Preliminaries}
\label{sec:prel}
We now briefly present several key lemmas for solving equations of the form $ax^4+bx^2y^2+cy^4=z^2$, for the proofs, see~\cite{Brown1989} and \cite{Polcklington1912}.



\begin{lemma}\label{thm:2}
All the positive integer solutions of the equation $xy=zt$ are given by the formulae
\[
x=ac, \quad y=bd, \quad z=ad, \quad t=bc.
\]
Moreover if we assume $(x,y)=(z,t)=1$, then $a, b, c$, and $d$ are pairwise coprime.
\end{lemma} 

\begin{lemma}\label{lm:1}
Consider the equation $z^2=Dy^2+x^2$ with $(x,Dy)=1$ or $(z,Dy)=1$. Then there exist integers $p, q, m$, and $n$ such that $(pm,qn)=1$ and such that:
\begin{enumerate}
	\item if $Dy$ is odd, then $pq=D$, $2x=pm^2-qn^2$, $y=mn$ and $2z=pm^2+qn^2$;
	\item if $Dy$ and $y$ are even, then $pq=D$, $x=pm^2-qq^2$, $y=2mn$, and $z=pm^2+qn^2$.
\end{enumerate}
\end{lemma}	
\section{main result}
\label{sec:main}
\textbf{Proof of Theorem \ref{thm:main-equation}}

Suppose to the contrary that $(\alpha, \beta, \gamma)$ is a solution of the equation  $k^2=4t^4-5t^2s^2+s^4$. We can suppose $(\alpha, \beta)=1$, since in the case $(\alpha, \beta)=d>1$ we would have $\alpha=d\alpha_1$, $\beta=d\beta_1$, substitute them in the given equation we obtain $\gamma^2=d^4\alpha_1^4-5d^4\alpha_1^2\beta_1^2+d^4\beta_1^4$, which implies $d^2|\gamma$, so $\gamma=d^2\gamma_1$. Hence $(\alpha_1,\beta_1,\gamma_1)$ is a solution of the equation $k^2=4t^4-5t^2s^2+s^4$, where $(\alpha_1,\beta_1)=1$. On the other hand we can see if $(\alpha, \beta, \gamma)$ is a non-trivial solution of the given equation, then $(\gamma, \alpha \beta, |4\alpha^4-\beta^4|)$ is a non-trivial solution of the equation~$z^2=x^4+10x^2y^2+9y^4$ \cite[Lemma $1$]{Dolan2012} which would contradict Lemma~\ref{lm:4}. Therefore, the equation $k^2=4t^4-5t^2s^2+s^4$ has no solution in positive integers, moreover this fact implies Theorem \ref{thm:main}.


\begin{lemma}\label{lm:4}
The equation $z^2=x^4+10x^2y^2+9y^4$ has no solution in positive integers.
\end{lemma}
\begin{proof}
We will prove this by contradiction. Suppose that there exists a solution $x, y$ and $z$ of the equation $z^2=x^4+10x^2y^2+9y^4$ in positive integers. Without loss of generality we can assume $(x,y)=1$. Indeed if $(x,y)=d>1$, writing $x=dx_1, y=dy_1$ for some integers $x_1$ and $y_1$. Substituting them in the equation $z^2=x^4+10x^2y^2+9y^4$ we obtain $z^2=d^4(x_1^4+10x_1^2y_1^2+9y_1^4)$, thus $d^2|z$ and so~$z=d^2z_1$. Therefore $x_1, y_1$, and $z_1$ is a solution of the given equation with $(x_1,y_1)=1$.

Observing that any perfect square leaves a residue of $0$ or $1$ module $3$. Taking remainders modulo $3$, we see that if $x$ and $y$ are not divisible by~$3$ we get $x^2 \equiv y^2 \equiv 1 \mod 3$, this implies
\[
z^2=x^4+10x^2y^2+9y^4 \equiv 1+ 10+ 9 \equiv 2 \mod 3,
\]
which is not possible. Thus $3$ must be a divisor of one of the numbers $x$ or $y$. We will consider two cases:


\textbf{Case~$1$: if $3|y$, $(x,y)=1$:}

The equation $z^2=x^4+10x^2y^2+9y^4$ is equivalent to $z^2=(x^2-3y^2)^2+(4xy)^2$. 
	
\noindent\textbf{Claim $(x^2-3y^2,4xy)=1$}
	
Let $(x^2-3y^2, 4xy)=d$, and let $p$ be a prime divisor of $d$. 
\begin{enumerate}
	\item If $p>3$, since $p|(x^2-3y^2)$ and $p|4xy$, $p$ must divide both $x$ and $y$ which contrasts with $(x,y)=1$.
	\item if $p=3$, since $3y^2\equiv 0 \mod 3$, $3$ must be a divisor of $x$ as well, while $(x,y)=1$.
\end{enumerate}
Therefore, $d$ must be a power of $2$. Now if $d=2^a$ and $a \ge 2$, then $4| x^2-3y^2$, it is equivalent to say $x^2 \equiv 3y^2 \mod 4$ and we can easily check that it is not possible. Indeed notice that squares can only be congruent to $0$ or $1$ modulo $4$. If $x$ and $y$ are both odd numbers we obtain $1 \equiv 3 \mod 4$, if $x$ odd and $y$ even we obtain $1 \equiv 0 \mod 4$, and finally if $x$ even and $y$ odd we obtain $0 \equiv 3 \mod 4$, which we get a contradiction in each case. Therefore $(x^2-3y^2,4xy)\le 2$. 
In the next step, we show that $2$ also is not possible. Suppose $(x^2-3y^2,4xy)=2$, then by Pythagorean triples formula there exist two coprime positive integers $a$ and $b$, of opposite parity, such that\footnote{since $a$ and $b$ are of opposite parity, $a^2-b^2$ is odd, hence $4xy= 2(a^2-b^2)$ is not possible.}
\[
x^2-3y^2=2(a^2-b^2),\quad 4xy=2(2ab), \quad z=2(a^2+b^2),
\]
which implies that $xy=ab$. Therefore, either $x$ or $y$ must be even. If $x$ is even, since $2|x^2-3y^2$, $y$  must be even as well, this is a contradiction. The same argument shows that $y$ also can not be even. Thus $(x^2-3y^2,4xy)=1$ and by primitive Pythagorean triples formula there exist two coprime positive integers $m$ and $n$, of opposite parity, such that
\begin{equation}\label{eq:1}
x^2-3y^2=m^2-n^2, \quad 4xy= 2mn, \quad z=m^2+n^2.
\end{equation}
In particular, the numbers $x$ and $y$ can not be odd simultaneously, otherwise, $x^2-3y^2$ is even which implies~$(x^2-3y^2, 4xy) \not=1$, and we have seen it is not possible. Now we will show that both numbers $x$ and $y$ can not have opposite parity. 
	
\textbf{$x$ odd and $y$ even}
	
First if $x$ is odd and $y$ even, then the number $x^2-3y^2$ and hence $m^2-n^2$ is of the form $4k+1$, which implies $m$ is odd and $n$ is even. Since $4xy=2mn$ and $n=2n_1$, we have $4xy=4mn_1$ so $xy=mn_1$. Therefore by Lemma \ref{thm:2} there exist natural numbers $a,b,c$ and $d$ 
such that 
\[
x=ac, \quad y=bd, \quad m=ad, \quad n_1=bc,
\]
where from $(x,y)=(m,n_1)=1$ we conclude any two of the numbers $a, b, c$, and $d$ are pairwise coprime. Since $x$ and $m$ are odd numbers we have $a,c$ and $d$ are odd and since $y$ is even we have $b$ is even. Replacing $x=ac$, $y=bd$, $m=ad$ and $n_1=bc$ in the equation $x^2-3y^2=m^2-n^2$, we obtain $(a^2+4b^2)c^2=(a^2+3b^2)d^2$. 

Let $r=(a^2+4b^2, a^2+3b^2)$, then $r|a^2+3b^2$ and $r|a^2+4b^2$, hence $r|b^2$. On the other hand from $r| 4(a^2+3b^2)-3(a^2+4b^2)$ we obtain $r|a^2$. Since $(a,b)=1$ we have $r=1=(a^2+4b^2, a^2+3b^2)$. Hence the equation $(a^2+4b^2)c^2=(a^2+3b^2)d^2$ together $(c,d)=1$ imply 
\[
a^2+(2b)^2=d^2, \quad a^2+3b^2=c^2.
\] 
Since $(a,b)=1$ and $a$ is odd, we have $(a,2b)=1$, and by primitive Pythagorean triples formula there exist two coprime integers $x_1$ and $y_1$ of opposite parity such that $2b=2x_1y_1$ and $a=x_1^2-y_1^2$, substituting them in the equation $a^2+3b^2=c^2$, we obtain $(x_1^2-y_1^2)^2+3(x_1y_1)^2=c^2$, which implies a non-trivial solution of the equation $c^2=x_1^4+x_1^2y_1^2+y_1^2$ which would contradict \cite[Theorem $2.4.1$]{Korpal2015}.
	
\textbf{$x$ even and $y$ odd}
	
Now we show that the case $x$ even and $y$ odd also has no solution. Suppose the contrary, let $x$ be an even number and $y$ an odd number such that the pair $x, y$ gives a non-trivial solution of the equation $z^2=x^4+10x^2y^2+9y^4$. By the least integer principle we may assume~$xy$ has the least possible value. Since $x$ is even and $y$ is odd, from equation~\eqref{eq:1}  we obtain $x^2-3y^2$ and hence $m^2-n^2$ is of the form $4k+1$, which shows that $m$ is odd and $n$ is even.  Since $4xy=2mn$ and $n=2n_1$, we have $xy=mn_1$. Therefore by Lemma~\ref{thm:2} there exist pairwise coprime natural numbers $a,b,c$ and $d$ such that 
\[
x=ac, \quad, y=bd, \quad m=ad, \quad n_1=bc,
\]
since $y$ and $m$ are odd numbers, thus $a, b$ and $d$ are odd numbers and since $x$ is even we have $a$ is even. Substituting these values in the equation  $x^2-3y^2=m^2-n^2$ and if we continue as above we obtain 
\[
a^2+(2b)^2=d^2, \quad a^2+3b^2=c^2.
\]
In particular, since $(a,c)=(b,c)=(a,b)=1$, therefore $(a,3b)=1$ in particular $3b$ is odd. Applying Lemma \ref{lm:1} for the equation $a^2+3b^2=c^2$ we obtain two positive integers $x_1$ and $y_1$ such that $2a=x_1^2-3y_1^2$ and $b=x_1y_1$. Substituting them in the equation $d^2=a^2+(2b)^2$ implies $d^2=x_1^4+10x_1^2y_1^2+9y_1^4$. On the other hand
\[
x_1y_1=b \le bd=y< xy.
\] 
This contradicts the assumed minimality of $xy$. Thus the equation $z^2=x^4+10x^2y^2+9y^4$ has no solution $x, y$, and $z$ with $3|y$.
	
\textbf{Case $2$: if $3|x$ and $(x,y)=1$}

Let $x,y$, and $z$ be a solution of the equation $z^2=x^4+10x^2y^2+9y^4$, such that~$3|x$. Writing~$x=3x_1$ we obtain $z^2=81x_1^4+90x_1^2y^2+9y^4$, thus $9|z$ and so $z=3z_1$. Dividing both sides by $9$ we get $z_1^2=y^4+10y^2x_1^2+9x^4_1$ which means $(t,s,k)=(y,x_1,z_1)$ is a solution of the equation~$k^2=t^4+10t^2s^2+9s^4$ in positive integers.
 On the other hand, as we have seen either $x_1$ or $y$ must be divisible by $3$, since~$(3,y)=1$ we conclude $3|x_1$, where by the argument of case $1$ it is impossible. This complete the proof.
\end{proof}
\section{Acknowledgement}
The author was supported by Grant Artin Approximation, Arc-R{\"a}ume, Aufl{\H o}sung von Singularit{\"a}ten FWF P-31336 and the Austrian Science Fund FWF Project P~30405-N32. I would like to thank Matteo Gallet for his feedback on an earlier draft of the paper.

\providecommand{\bysame}{\leavevmode\hbox to3em{\hrulefill}\thinspace}
\providecommand{\MR}{\relax\ifhmode\unskip\space\fi MR }
\providecommand{\MRhref}[2]{%
	\href{http://www.ams.org/mathscinet-getitem?mr=#1}{#2}
}
\providecommand{\href}[2]{#2}

\end{document}